\documentclass[dvipsnames,12pt]{article}
\usepackage{amssymb, latexsym,pdfsync,amsmath,amsthm, ulem,hyperref,graphicx,tcolorbox}
\usepackage{fullpage}
\usepackage{pgf,tikz}
\usepackage{mathrsfs}
\usetikzlibrary{arrows}
\newtheorem{theorem}{Theorem}

\newtheorem{lemma}{Lemma}
\newtheorem{proposition}{Proposition}

\theoremstyle{definition}
\newtheorem{definition}{Definition}
\newtheorem{remark}{Remark}

\newtheorem*{Kantor}{Kantor's Conjecture}

\newcommand{\SSS}{\mathbb{S}}

\newcommand{\FF}{\mathbb{F}}
\newcommand{\Fp}{\mathbb{F}_p}
\newcommand{\Fq}{\mathbb{F}_q}
\newcommand{\Fqn}{\mathbb{F}_{q^n}}
\newcommand{\Fqm}{\mathbb{F}_{q^m}}

\newcommand{\HH}{\mathcal H}
\newcommand{\B}{\mathcal B}

\newcommand{\G}{\mathcal G}
\newcommand{\C}{\mathcal C}
\newcommand{\cP}{\mathcal P}
\newcommand{\cS}{\mathcal S}

\newcommand{\cI}{\mathcal I}

\def\F{\mathbb{F}}

\def\Fq{{\mathbb{F}}_q}

\def\Hom{\mathrm{Hom}}
\def\Bil{\mathrm{Bil}}
\def\Aut{\mathrm{Aut}}
\def\im{\mathrm{im}}

\def\PG{\mathrm{PG}}

\def\GL{\mathrm{GL}}

\def\dim{\mathrm{dim}}

\def\L{\mathcal{L}}

\def\Tr{\mathrm{Tr}}

\def\tr{\mathrm{tr}}

\def\Gal{\mathrm{Gal}}

\def\I{\mathcal{I}}
\def\a{\mathbf{\alpha}}

\newcommand{\pt}[1]{\langle #1 \rangle}

\DeclareMathOperator{\modr}{~mod_r}

\newcommand{\npmatrix}[1]{\left( \begin{matrix} #1 \end{matrix} \right)}

\newcommand{\rank}{\mathrm{rank}}

\begin{document}
\title{MRD Codes: Constructions and Connections}
\author{John Sheekey}
\date{\today}
\maketitle

{\bf This preprint is of a chapter to appear in {\it Combinatorics and finite fields: Difference sets, polynomials, pseudorandomness 
and applications. Radon Series on Computational and Applied Mathematics}, K.-U. Schmidt and A. Winterhof (eds.). The tables on classifications  will be periodically updated {\color{blue}(in blue)} when further results arise. If you have any data that you would like to share, please contact the author.} 

\begin{abstract}

{\it Rank-metric codes} are codes consisting of matrices with entries in a finite field, with the distance between two matrices being the rank of their difference. Codes with maximum size for a fixed minimum distance are called {\it Maximum Rank Distance} (MRD) codes. Such codes were constructed and studied independently by Delsarte (1978), Gabidulin (1985), Roth (1991), and Cooperstein (1998). Rank-metric codes have seen renewed interest in recent years due to their applications in {\it random linear network coding}.

MRD codes also have interesting connections to other topics such as {\it semifields} (finite nonassociative division algebras), finite geometry, linearized polynomials, and cryptography. In this chapter we will survey the known constructions and applications of MRD codes, and present some open problems.
\end{abstract}

\section{Definitions and Preliminaries}
\subsection{Rank-metric codes}

Coding theory is the branch of mathematics concerned with the efficient and accurate transfer of information. {\it Error-correcting codes} are used when communication is over a channel in which errors may occur. This requires a set equipped with a {\it distance function}, and a subset of allowed {\it codewords}; if errors are assumed to be small, then a received message is decoded to the nearest valid codeword.

The most well-known and widely-used example are codes in the {\it Hamming metric}; codewords are taken from a vector-space over a finite field, and distance between two vectors in $\Fq^n$ is defined as the number of positions in which they differ. We refer to \cite{MacWSl} for a detailed reference on this well-studied topic.

In rank-metric coding, codewords are instead taken from the set of matrices over a finite field, with the distance between two matrices defined as the rank of their difference. 

\begin{definition}
Let $\Fq$ denote the finite field with $q$ elements, and $M_{n \times m}(\Fq)$ the set of $n \times m$ matrices with entries in $\Fq$, with $m\leq n$. The {\it rank-weight} of a matrix $X$ is defined to be its usual column rank, and denoted by $\rank(X)$. The {\it rank-distance} between two matrices $X,Y\in M_{n \times m}(\Fq)$ is denoted by $d(X,Y)$ and defined as 
\begin{equation*}
d(X,Y) = \rank(X-Y).
\end{equation*}
\end{definition}

A {\it rank-metric code} is then a subset $\C$ of $M_{n \times m}(\Fq)$. The {\it minimum distance} of a set $\C$ with at least two elements, denoted $d(\C)$ is
\[
d(\C) = \min\{d(X,Y):X,Y \in \C,X\ne Y\}.
\]

\subsection{MRD codes}
A main goal in coding theory is to find codes of maximum possible size for a given minimum distance. While there are a variety of bounds for codes in the Hamming metric \cite{MacWSl}, for rank metric codes we need only the following simple analogue of the {\it Singleton bound}, first proved by Delsarte \cite{Delsarte1978}. We include a short proof for illustrative purposes.

\begin{theorem}[Singleton-like bound, Delsarte]
Suppose $\C\subset M_{n \times m}(\Fq)$ is such that $d(\C) = d$. Then
\begin{equation}
|\C|\leq q^{n(m-d+1)}.
\end{equation}
\end{theorem}

\begin{proof}
Suppose $|\C|> q^{n(m-d+1)}$. By the pigeonhole principle, there must exist two distinct elements $X,Y$ of $\C$ which coincide in every entry of their first $m-d+1$ rows. Thus $X-Y$ has at least $m-d+1$ zero rows, and thus $\rank(X-Y)\leq d-1<d$, a contradiction.
\end{proof}

\begin{definition}
A subset $\C$ of $M_{n \times m}(\Fq)$ satisfying $|\C|= q^{n(m-d(\C)+1)}$ is said to be a {\it Maximum Rank Distance code}, or {\it MRD-code} for short. We say $\C$ is a $[n \times m,d(\C)]$-MRD code.
\end{definition}

MRD-codes are the rank-metric analogue of MDS-codes in the Hamming metric. Unlike MDS codes, MRD codes in fact exist for all values of $q,n,m$, and all $d$. This was first shown by Delsarte \cite{Delsarte1978}; indeed, he showed that the bound can be met by an additively-closed set. 

\begin{theorem}[\cite{Delsarte1978}]
There exists an MRD code in $M_{n \times m}(\Fq)$ with minimum distance $d$ for all $q,n,m$, and all $d$.
\end{theorem}

This article with focus on the case where $\C$ is an MRD-code which is additively-closed, which we will introduce in Section \ref{ssec:additive}. We will recall Delsarte's construction in Section \ref{sec:constructions}.

\begin{remark}
There are many equivalent ways of representing rank-metric codes. Fundamentally we are considering subsets of the space of homomorphism from one vector space to another, denoted by $\Hom(V,W)$, or the space of bilinear forms from the product of two vector spaces to the underlying field, denoted by $\Bil(V\times W,Fq)$. In the case of finite-dimensional vector spaces, we often choose to work with matrices due to their familiarity. However other representations are frequently used in the literature. We will outline these representations and their equivalences in Section \ref{sec:representations}. 


The most commonly used are vectors in $(\Fqn)^m$, and linearized polynomials in $\Fqn[x]$. 
\end{remark}

\begin{definition}
\label{def:vec}
For a vector $v= (v_1,v_2,\ldots,v_m)$, its rank-weight is defined as the dimension of the $\Fq$-span of its entries, i.e.
\[
\omega(v) = \dim_{\Fq}\pt{v_1,v_2,\ldots,v_m}.
\]
Then the rank-distance between two vectors $u,v\in (\Fqn)^m$ is defined as
\[
d(u,v) = \omega(u-v).
\]
\end{definition}
Many recent results have been obtained by representing rank-metric codes as sets of {\it $\sigma$-linearized polynomials}. These have the advantage of giving results for a more general class of fields, namely fields with a cyclic Galois extension (i.e. whose Galois group is a cyclic group of order equal to the degree of the extension).
\begin{definition}
\label{def:linpoly}
Let $K$ be a field, and $L$ a cyclic Galois extension of degree $n$, and $\sigma$ a generator for $\Gal(L:K)$. A {\it $\sigma$-linearized polynomial} is an expression of the form
\[
f(x) = f_0x+f_1x^{\sigma}+\cdots+f_kx^{\sigma^k},
\]
where $f_i\in L$, and $k$ is some nonnegative integer. We denote the set of all such expressions as $\L_\sigma$. If $k$ is the largest integer such that $f_k\ne 0$, we say that $k$ is the {\it $\sigma$-degree} of $f$.
\end{definition}
If $K=\Fq$, $L=\Fqn$, then $a^\sigma = a^{q^s}$ for all $a\in \Fqn$, and so a $\sigma$-linearized polynomial is a linearized polynomial in the usual sense \cite[Section 3.4]{Lidl}. A $\sigma$-linearized polynomial defines a $K$-linear map from $L$ to itself, and by \cite{GoQu2009b}, every such map can be represented by a unique $\sigma$-linearized polynomial of degree at most $n-1$. 
\begin{definition}
The {\it rank} and {\it nullity} of a $\sigma$-linearized polynomial $f$ are defined as the rank and nullity of the $K$-linear map from $L$ to itself determined by $f$.
\end{definition}

From here on we will refer to a rank-metric code without specifying the representation unless necessary.

\subsection{Linearity and Generator Matrices}
\label{ssec:additive}
\begin{definition}
A rank-metric code is {\it additive} if it is closed under addition; i.e. $a,b\in \C \Rightarrow a+b\in \C$.
\end{definition}
Clearly $\C$ being additive coincides precisely with it being a subspace over the prime field. A set of matrices may also be a subspace over a larger field, for example $\Fq$. In fact, either of the representations of Definition \ref{def:vec} or Definition \ref{def:linpoly} in fact allow us to consider codes to be linear over fields larger than $\Fq$.
\begin{definition}
A rank-metric code is {\it $F$-linear} if it is additive and closed under $F$-multiplication, for some subfield $F$ of $\FF_{q^{n}}$.
\end{definition}
Note that what we mean by ``$F$-multiplication'' will depend on the representation chosen. We will clarify this further in Section \ref{sec:idealisers}.


Suppose $\C$ is an $\Fqn$-linear code. Then we can view $\C$ as an $\Fqn$-subspace of $(\Fqn)^m$, and represent $\C$ by a {\it generator matrix}.
\begin{definition}
Let $\C$ be an $\Fqn$-linear code in $(\Fqn)^m$ with $\Fqn$-dimension $k$. Then a $k\times m$ matrix $G$ with entries in $\Fqn$ is called a {\it generator matrix} for $\C$ if its rows form an $\Fqn$-basis for $\C$.
\end{definition}

For additive codes, we have the following simple extension of the Singleton-like bound to codes over any field which is finite-dimensional over its prime subfield. 
\begin{theorem}
Suppose $\C\subset M_{n \times m}(K)$ is additive, and $d(\C) = d$. Let $E$ be the prime subfield of $K$, and $e=[K:E]$. Then
\begin{equation*}
\dim_E(\C)\leq en(m-d+1).
\end{equation*}
\end{theorem}
The proof follows from a simple dimension argument. If $\dim_E(\C)> en(m-d+1)$, then $\C$ must intersect nontrivially with the subspace of matrices with first $m-d+1$ rows all zero, and thus contain an element of rank at most $d-1$.

\subsection{Semifields}

A finite {\it presemifield} is a division algebra with a finite number of elements in which multiplication is not necessarily associative; if a multiplicative identity element exists, it is called a {\it semifield}. If we only assume left-distributivity of multiplication over addition, we have a {\it (left) quasifield}. Semifields were first studied by Dickson \cite{Dickson1905}, \cite{Dickson1906} in the early 1900s. The topic was further developed by Albert \cite{Albert1961} and Knuth \cite{Knuth1965} in the 1960s. This early research was motivated by the fact that semifields give rise to a certain class of projective planes. Much recent research on semifields has been performed due to surprising connections in finite geometry. We refer to \cite{Dembowski}, \cite{Handbook}, \cite{LaPo2011}, \cite{Kantor2006} for further background and references.

If $\SSS$ is a (pre)semifield, $n$-dimensional over $\Fq$, then we may identify the additive structure of $\SSS$ with the vector space $(\Fq)^n$. Let us denote multiplication in $\SSS$ by $\circ$. Then right-multiplication by a nonzero element of $\SSS$ defines an invertible $\Fq$-linear map from $(\Fq)^n$ to itself;
\[
R_y(x) := x\circ y.
\]
This is an $\Fq$-linear map because $\SSS$ is an algebra, and is invertible because $\SSS$ is a division algebra. Thus, by choosing a basis, we may view $R_y$ as an element of $M_{n\times n}(\Fq)$. Taking the set of all such $R_y$ gives a subspace of $M_{n\times n}(\Fq)$. In the literature on semifields, this is known as a {\it semifield spread set}. As every nonzero element is invertible, it coincides precisely with the definition of an $\Fq$-linear MRD code with minimum distance $n$. 

Two presemifields $\SSS=(\Fq^n,\circ)$ and $\SSS'=(\Fq^n,\circ')$ are said to be {\it isotopic} if there exist invertible additive maps $X,Y,Z$ on $\Fq^n$ such that for all $x,y\in \SSS$ it holds that 
\[
X(x\circ y) = Y(x)\circ'Z(y).
\]
Isotopism is considered the natural notion of equivalence for semifields, in part due to the fact that two semifields are isotopic if and only if the projective planes they co-ordinatise are isomorphic.

The following gives the well-known correspondence between semifields and MRD codes. See for example \cite{Handbook}, \cite{LaPo2011}, \cite{Willems}.

\begin{theorem}
Finite left quasifields of dimension $n$ over $\Fq$ are in one-to-one correspondence with MRD-codes in $M_{n\times n}(\Fq)$ with minimum distance $n$.

Finite presemifields of dimension $n$ over $\Fq$ are in one-to-one correspondence with $\Fq$-linear MRD-codes in $M_{n\times n}(\Fq)$ with minimum distance $n$.

Two semifields are isotopic if and only if their semifield spread sets are equivalent.

Isotopy classes of finite semifields of dimension $n$ over $\Fq$ are in one-to-one correspondence with equivalence classes of $\Fq$-linear MRD-codes in $M_{n\times n}(\Fq)$ with minimum distance $n$.
\end{theorem}

The study of semifields predates the study of MRD codes by over half a century, so it is perhaps unsurprising that there are more results on semifields than MRD codes. Indeed, many of the recent constructions for families of MRD codes were inspired by known constructions for semifields. We will outline some of these constructions in Section \ref{sec:constructions}.

%

\subsection{Subspace codes}
Some of the motivation for the recent interest in rank-metric codes is due to their connections to subspace codes. Subspace codes have applications in random network coding, due to the influential work of Koetter-Kchishang \cite{KoKs2008}, \cite{SiKsKo2008}. We refer also to \cite{NetworkBook}, \cite{Etzion} for further details and open problems.

A {\it subspace code} is a set of subspaces of a vector space $\Fq^N$, with distance defined by the {\it subspace distance}:
\[
d(U,V) := \dim(U+V)-\dim(U\cap V) = \dim(U)+\dim(V)-2\dim(U\cap V).
\]
Rank-metric codes in $M_{n \times m}(\Fq)$ give rise to {\it constant-dimension codes} in $\Fq^{m+n}$ via the process known as {\it lifting}. From a matrix $A\in M_{n \times m}(\Fq)$ we define the $n$-dimensional subspace $S_A$ of $\Fq^{m+n}$ by
\[
S_A = \{(x,Ax):x\in \Fq^m\}.
\]
Then the distance between two lifted subspaces is determined by the rank-distance;
\[
d(S_A,S_B) = 2\rank(A-B).
\]
From a code $\C$, we define the lifted code $S_\C := \{S_A:A\in \C\}$. Then a rank-metric code with high minimum distance gives rise to a subspace code with high minimum distance. However, lifted MRD codes are not usually optimal as subspace codes, except in the case of $m=n=d$, in which case we have the well-known correspondence between {\it spreads} and {\it spread sets}. Lifted MRD codes can often be extended to larger subspace codes, but again not always to optimal subspace codes. Some recent results on this can be found in \cite{CoPa}, \cite{Heinlein}, with further background and open problems in \cite{Etzion}.

\subsection{Segre Varieties and their secant varieties}
Geometrically speaking, we can view an $\Fq$-linear rank-metric code in $M_{n \times m}(\Fq)$ as a subspace in the projective space $\PG(mn-1,q)$, and an $\Fqn$-linear code as a subspace in $\PG(m-1,q^n)$. The set of matrices of rank one corresponds to a Segre variety in the former case, and a subgeometry in the latter case. Equivalence of $\Fq$-linear (resp. $\Fqn$-linear) rank-metric codes then corresponds to equivalence under the setwise stabiliser of the Segre variety (resp. subgeometry) inside the collineation group of the projective space.

The set of elements of rank at most $i$ then correspond to the $(i-1)$-st secant variety, and an MRD code corresponds to a maximal subspace disjoint from one of these secant varieties. We refer the interested reader to \cite{LavSem}, \cite{LunardonMRD}, \cite{SheekeyVdV} for further development of this correspondence.

\section{Properties and Operations for MRD codes}

\subsection{Rank Distribution}

Given a set $\C\subset M_{n \times m}(\Fq)$, we define the {\it rank-distribution} of $\C$, denoted $r(\C)$, by the $n$-tuple of nonnegative integers $(a_0,a_1,\ldots,a_n)$, where
\[
a_i = \#\{X\in \C:\rank(X) = i\}.
\]
We define the {\it weight enumerator} of $\C$, denoted by $W_{\C}(x,y)$, as the polynomial
\[
W_{\C}(x,y) = \sum_{i=0}^n a_i x^iy^{m-i}.
\]
Delsarte \cite{Delsarte1978} proved the following, using association schemes. Gabidulin \cite{Gab1985} gave a different proof in the case of $\Fqn$-linear codes). More elementary recent proofs can be found in \cite{Gadouleau}, \cite{DuGoMcGSh}.
\begin{theorem}[\cite{Delsarte1978}]
The rank distribution of an additive MRD code is completely determined by its parameters $q,m,n,$ and $d$. In particular, if $\C$ is an MRD code in $M_{n \times m}(\Fq)$ with minimum distance $d$, then
\[
a_i = {m\brack i} \sum_{s=0}^{i-d}(-1)^{{s\choose 2}}{i\brack s}(q^{n(d+i-s+1)}-1).
\]

\end{theorem}

\subsection{Equivalence and Automorphism groups}
\label{sec:idealisers}

Multiplying a matrix on the left or right by an invertible matrix does not change its rank. Given a set $\C\subset M_{n\times m})(\Fq)$, and a pair of matrices $X,Y$, we define
\[
X\C Y = \{ XA Y:A\in \C\}.
\]
Given an automorphism of $\Fq$ and a matrix $A$ with $(i,j)$-entries $A_{ij}\in\Fq$, we define $A^\rho$ as the matrix with $(i,j)$-entry $A_{ij}^\rho$. 

\begin{definition}
Two codes $\C_1$, $\C_2$ are said to be {\it equivalent} if there exist $X\in \GL(n,q),Y\in \GL(m,q)$, and $\rho\in \Aut(\Fq)$ such that $\C_2 =X\C_1^\rho Y$. If $\rho=1$ we say they are {\it properly equivalent}.
\end{definition}

\begin{definition}
The {\it automorphism group} of $\C$ is denoted by $\Aut(\C)$ and defined by $\Aut(\C) =\{(X,Y,\rho):X\C^\rho Y=\C\}$. The {\it proper automorphism group} is denoted by $\overline{\Aut}(\C)$.
\end{definition}

The corresponding action on a linearized polynomial $f$ by a pair of linearized polynomials $g,h$ and an automorphism $\rho$ is 
\[
f \mapsto g \circ f^\rho \circ h \mod (x^{\sigma^n}-x),
\]
where $f^\rho(x) = \sum_{i=0}^{m-1} f_i^{\rho} x^{\sigma^i}$. Note that $f^\rho(x) = (f(x^{\rho^{-1}}))^\rho$ for all $x$.

The corresponding action on a vector $v\in \Fqn^m$ is by a linearized polynomial $g$, a matrix $Y\in \GL(m,q)$, and an automorphism $\rho$ as follows:
\[
v \mapsto (g(v_1)^\rho,\cdots, g(v_m)^\rho) Y.
\]
\begin{definition}
The {\it left-idealiser} of a code $\C$ is defined as
\[
\I_{\ell}(\C) = \{X:X \in M_{n\times n}(\Fq), X\C\subseteq \C\}
\]
The {\it right idealiser} $\I_r(\C)$ is defined as
\[
\I_r(\C) = \{Y:Y \in M_{m\times m}(\Fq), \C Y\subseteq\C\}
\]
\end{definition}
If $\C$ is an additive MRD code with $m=n=d$, i.e. a semifield spread set, then the left- and right-idealisers are isomorphic to the {\it left- and middle-nuclei}. This in fact predates the notion of the idealisers of a code, but we adopt the terminology of idealisers here. Development of these ideas is due to \cite{MaPoNuc}, \cite{LiebNebe}, \cite{TrZh}, \cite{LuTrZhKernel}. These concepts are very useful when determining the newness of constructions.

We can now directly determine when a code consisting of matrices is equivalent to an $\Fqn$-linear code in $(\Fqn)^m$.

\begin{definition}
A code $\C$ in $M_{n \times m}(\Fq)$ is said to be $\Fqn$-linear if and only if its left-idealiser contains a subring isomorphic to $\Fqn$.
\end{definition}

Since all subspaces of $M_{n\times n}(\Fq)$ isomorphic to $\Fqn$ are similar, given two $\Fqn$-linear codes we may assume without loss of generality that their left-idealisers contain the same subspace isomorphic to $\Fqn$. From now on we will fix such a subspace, and by abuse of notation refer to it as $\C(\Fqn)$. This can be explicitly realised as the subring generated by the companion matrix of an irreducible polynomial of degree $n$ over $\Fq$.

The set of equivalences preserving $\Fqn$-linearity are those for which the matrix $Y$ is in the normaliser of $\C(\Fqn)$, which we denote by $N(\C(\Fqn))$. It is well known that this is a group isomorphic to the semidirect product of $\Fqn^*$ with a cyclic group of order $n$.

In terms of linearized polynomials, this corresponds to having $g(x) = \alpha x^{q^i}$, where $\alpha \in \Fqn^*$.  

\begin{definition}
Two codes $\C_1$, $\C_2$ with left-idealiser containing $\C(\Fqn)$ said to be {\it semi-linearly equivalent} if there exist $X\in N(\C(\Fqn)),Y\in \GL(m,q)$, and $\rho\in \Aut(\Fq)$ such that $\C_2 =  X\C_1^\rho Y$.
\end{definition}

\begin{theorem}[\cite{LiebNebe}]
Let $\C$ be an $\Fq$-linear code in $M_{n \times m}(\Fq)$. Then
\[
\overline{\Aut}(\C)\subset N(\cI_\ell(\C))\times N(\cI_r(\C)),
\]
where $N(\cI_\ell(\C))$ denotes the normaliser of $\cI_\ell(\C)$ in $\GL(m,q)$, and $N(\cI_r(\C))$ denotes the normaliser of $\cI_r(\C)$ in $\GL(n,q)$.
\end{theorem}

This implies that for an $\Fqn$-linear code, the proper automorphism group is contained in $\GL(m,q)\times (\Fqn^*\times C_n)$.

\begin{theorem}[\cite{SheekeyVdV}]\label{thm:ShVdV}
Two $\Fqn$-linear codes are semilinearly equivalent if and only if they are equivalent.
\end{theorem}

\begin{remark}
The notion of equivalence for MRD codes defined by semifields was well established, arising from the definition of isotopy and isomorphism for the related projective planes. We refer to \cite{LaPo2011}, \cite{Handbook} for further details.

A detailed treatment of equivalence and semilinear equivalence for rank-metric codes can be found in \cite{Morrison}, building on \cite{Berger}, \cite{HuaWan}. Some issues remained in the interplay between the two notions of equivalence, which we believe has now been resolved by Theorem \ref{thm:ShVdV}.
\end{remark}

\begin{remark}
If $n=m$, then the transpose operation for matrices also preserves rank. For this reason it is sometimes included in the definition of equivalence. Note also that some literature omits the action of $\Aut(\Fq)$ in the definition of equivalence.

When considering additive MRD codes with $n=m=d$, i.e. semifield spread sets, there is a further action of $S_3$ on such codes which preserves the MRD property, arising from the theory of {\it nonsingular tensors} \cite{Knuth1965}. The set of equivalence classes under this action is known as the {\it Knuth orbit}.
\end{remark}

\subsection{Puncturing and shortening}

Puncturing and shortening are well-known operations on codes in the Hamming-metric. There exist natural rank-metric analogues, which we present here. For an in-depth treatment of this we refer to \cite{ByrneRavagnani}. If we view a matrix as a map between vector spaces, then by restricting the domain and/or range to a subspace, we can obtain a matrix of smaller size. As we can view a matrix either as acting on row vectors or column vectors, there are various ways to do this. 

Given $A\in M_{n \times m}(\Fq)$, we define the {\it right-image} of $A$ by $\im_r(A) := \{Av:v\in \Fq^m\}\leq \Fq^n$. Similarly we define the {\it left-image} as $\im_l(A) := \{v^TA:v\in \Fq^n\}\leq \Fq^m$. Right- and left-kernels are similarly defined. For a subspace $U$ of $\Fq^k$ we define $U^* = \{v\in \Fq^k:v^Tu = 0~\forall u\in U\}$.

\begin{lemma}
For any subspace $U\leq \Fq^m$, any subspace $W\leq \Fq^n$, and any set $\C\subset M_{n\times m}(\Fq)$, we have that
\begin{align*}
\{A:A\in \C, \im_l(A)\leq U^*\} &= \{A:A\in \C, \ker_r(A)\geq U\}\\
\{A:A\in \C, \im_r(A)\leq W^*\} &= \{A:A\in \C, \ker_l(A)\geq W\}.
\end{align*}
\end{lemma}
Suppose $\C\subset M_{n \times m}(\Fq)$ is such that there exists an $(m-s)$-dimensional subspace $U$ of $\Fq^m$ with $\im_r(A)\subset U^*$ for all $A\in \C$. Then we may view $\C$ as a subset of $\Hom(\Fq^n,U^*)$, which is isomorphic to $M_{n\times s}(\Fq)$. 
\begin{definition}
Given a code $\C\subset M_{n \times m}(\Fq)$, an $(m-s)$-dimensional subspace $U$ of $\Fq^m$, and an $(n-t)$-dimensional subspace $W$ of $\Fq^n$, we define the {\it row-shortening by $U$} of $\C$ as the set
\[
\cS_{U,row}(\C) = \{A:A\in \C, \im_l(A)\leq U^*\}\subset\Hom(\Fq^m,U^*)\simeq M_{n\times s}(\Fq),
\]
and the {\it column-shortening by $W$} of $\C$
\[
\cS_{W,col}(\C) = \{A:A\in \C, \im_r(A)\leq W^*\}\subset\Hom(U,\Fq^m)\simeq M_{m\times t}(\Fq).
\]
\end{definition}

A proof of the following well-known fact can be found in \cite[Theorem 1]{DuGoMcGSh}.
\begin{theorem}
Let $\C$ be an $\Fq$-linear MRD code in $M_{n \times m}(\Fq)$ with minimum distance $d$. Then any row-shortening of $\C$ by an $(m-s)$-dimensional subspace of $\Fq^m$ is an MRD code in $M_{n\times s}(\Fq)$ with minimum distance $d$.
\end{theorem}
Note that the analogous statement for column-shortening is not necessarily true, unless $m=n$. It is important to recall when applying this theorem that we are assuming that $m\leq n$.

Suppose $\C\subset M_{n \times m}(\Fq)$, and $X\in M_{n\times n}(\Fq)$ is a matrix of rank $t$. Then the set $X\C$ is a set of matrices, each of whose right-image is contained in $\im_r(X)$. Thus we may view $X\C$ as a subset of $M_{t\times m}(\Fq)$.

\begin{definition}
A {\it row-puncturing} of a code $\C$ is the row-shortening of $X\C$ by $\im_r(X)$, and denoted by $\cP_{row,X}(\C)$. 

A {\it column-puncturing} of a code $\C$ is the column-shortening of $\C Y$ by $\im_l(Y)$, and denoted by $\cP_{col,Y}(\C)$. 
\end{definition}

\begin{remark}
Note that shortening changes the dimension but not the minimum distance, while puncturing changes the minimum distance but not the dimension. Both change the ambient space.

In \cite{ByrneRavagnani} puncturing and shortening are defined slightly differently. However the definitions are equivalent.

Another equivalent way to define puncturing is by the set $X\C$ for some $X\in M_{t\times n}(\Fq)$ or the set $\C Y$ for some $X\in M_{n\times s}(\Fq)$.
\end{remark}

\section{Constructions of MRD codes}
\label{sec:constructions}

\subsection{Semifields; $n=m=d$}

There are a wide variety of known constructions for semifields. We will not list them here; we instead refer to \cite{Kantor2006}, \cite{LaPo2011} for an overview. Many constructions are valid only in special cases; for example two-dimensional over a nucleus. Here we list a collection of constructions which are valid in many parameters.

{\bf Albert's generalised twisted fields \cite{Albert1961}.} Elements of $\Fqn$, with multiplication defined by $x\circ y = xy-cx^{p^i}y^{p^j}$, where $c$ is fixed with $N(c)\ne 1$. It gives semifields of arbitrary prime power order, excluding $2^e$ for prime $e$.

{\bf Petit's cyclic semifields  \cite{Petit}} (independently found by Jha-Johnson \cite{JHJO1989}). Elements are {\it skew-polynomials} over $\Fqn$ with degree at most $k-1$. Multiplication is defined by $a\circ b = ab \modr f$, where $f$ is an irreducible of degree $k$, and $\modr$ denotes remainder on right-division. It gives semifields of order $q^{ns}$, where $q=p^e$ for some $e>1$.

{\bf Knuth's binary semifields.} A family of commutative semifields in even characteristic. Kantor \cite{KantorComm} generalised this to a large family, which grows exponentially in the order, also in even characteristic. 

{\bf Pott-Zhou's family of commutative semifields \cite{PottZhou}.} A large family of commutative semifields of order $p^{2n}$.

\subsection{Delsarte-Gabidulin codes}

Let $K$ be a field, $L$ a cyclic Galois extension of degree $n$, and $\sigma$ a generator of $\Gal(L:K)$. Let $\G_{k,\sigma} $ denote the set of $\sigma$-linearized polynomials of degree at most $k-1$; i.e.
\[
\G_{k,\sigma} = \{ f_0x+f_1x^\sigma +\cdots+f_{k-1}x^{\sigma^{k-1}} : f_i\in L \}.
\]
The following lemma, \cite[Theorem 5]{GoQu2009b}, is key to the construction of the first and most well-known family MRD codes.
\begin{lemma}
\label{lem:nullity}
The nullity of a $\sigma$-linearized polynomial is at most its $\sigma$-degree.
\end{lemma}
Note that if $L=\Fqn$ and $x^\sigma = x^q$, then this lemma is a trivial consequence of the fact that a $\sigma$-linearized polynomial of degree $k$ is a polynomial of degree $q^k$.

Now $\dim_K(\G_{k,\sigma})=nk$, and by Lemma \ref{lem:nullity} every non-zero element of $\G_{k,\sigma}$ has nullity at most $k-1$, and hence rank at least $n-k+1$. Thus $\G_{k,\sigma}$ gives an MRD-code in $M_{n\times n}(\Fq)$ with minimum distance $d=n-k+1$. 

\begin{definition}
We call $\G_{k,\sigma}$, and any code obtained from $\G_{k,\sigma}$ by puncturing or shortening, a {\it Delsarte-Gabidulin code}.
\end{definition}

As $\G_{k,\sigma}$ is $\Fqn$-linear, its representation as vectors is an $\Fqn$-subspace of $(\Fqn)^m$. It has generator matrix has the form
\[
\npmatrix{\alpha_1&\alpha_2&\cdots&\alpha_m\\\alpha_1^{\sigma}&\alpha_2^{\sigma}&\cdots&\alpha_m^{\sigma}\\\vdots&\vdots&\ddots&\vdots\\\alpha_1^{\sigma^{k-1}}&\alpha_2^{\sigma^{k-1}}&\cdots&\alpha_m^{\sigma^{k-1}}},
\]
where the $\alpha_i$'s are elements of $\Fqn$, linearly independent over $\Fq$. Characterisations of Delsarte-Gabidulin codes in terms of their generator matrix can be found in \cite{HoMa}, \cite{Neri}.

\begin{remark}
Delsarte first discovered these in the case of $L=\Fqn$ a finite field and $x^\sigma = x^q$, using bilinear forms. Gabidulin independently discovered this using vectors in $(\Fqn)^m$, and again by Roth \cite{Roth1991} and Cooperstein \cite{CoopersteinM}. Much of the subsequent literature refers to these as {\it Gabidulin codes}. Gabidulin-Kshevetskiy extended this to the case $x^{\sigma} = x^{q^s}$, which became known as {\it generalised Gabidulin codes}. The case of general fields with cyclic Galois extension was shown in \cite{GoQu2009a}, and later independently in \cite{AuLoRo2013}. In \cite{LiebNebe} these were referred to as {\it Gabidulin-like codes}.

Since all of these can be constructed in a unified way requiring little extra theory, and are MRD for essentially the same reason, we will refer to all of these from now on as Delsarte-Gabidulin codes.
\end{remark}

%


The Delsarte-Gabidulin construction answered the question of the existence of MRD codes completely; such codes exist for any parameters. For many years this meant that searching for new MRD codes was not a priority. However in recent years, Delsarte-Gabidulin codes have proved unsuitable for some applications, in particular cryptography (see Section \ref{sec:crypto}), and so new constructions are required.


\subsection{MRD subcodes of Delsarte-Gabidulin codes}
In recent years, due to the necessity for new constructions for the purposes of applications, new constructions for additive MRD codes have been found. These all share the common theme of perturbing the Delsarte-Gabidulin codes by adding extra coefficients to the $\sigma$-linearized polynomials in such a way as to keep the rank high. All rely on the following important lemma \cite[Theorem 10]{GoQu2009a}.
\begin{lemma}
If the nullity of a $\sigma$-polynomial $f(x)$ of $\sigma$-degree $k$ is equal to $k$, then $N(f_0) = (-1)^{nk}N(f_k)$.
\end{lemma}
Thus we can find MRD codes contained in the Delsarte-Gabidulin code $\G_{k+1}$ by finding a $k$-dimensional $\Fqn$-subspace of $\G_{k+1}$ which avoids this condition on the first and last coefficients. The most straightforward way to do this is by finding additive functions $\phi_1,\phi_2$ with the below property.

\begin{proposition}
Suppose $\phi_1,\phi_2$ are two additive functions from $\Fqn$ to itself, and $k\leq n-1$. Let $\HH_k(\phi_1,\phi_2) $ be the set of $\sigma$-linearized polynomials
\[
\HH_k(\phi_1,\phi_2) = \left\{\phi_1(a)x+\left(\sum_{i=1}^{k-1}f_i x^{\sigma^i}\right)+\phi_2(a)x^{\sigma^k}:a,f_i\in\Fqn\right\}.
\]
If $N(\phi_1(a))\ne (-1)^{nk}N(\phi_2(a))$ for all $a\in \Fqn^*$, then $\HH_k(\phi_1,\phi_2)$ is an MRD code.
\end{proposition}
To date this result has been used to construct the following.
\begin{center}
\begin{tabular}{|c|c|c|c|c|c|c|}
\hline
Name &$\sigma$&$\phi_1(a)$&$\phi_2(a)$&Conditions&Semifield&Reference\\
\hline
TG&$q$&$a$&$\eta a^{q^h}$&$N(\eta)\ne(-1)^{nk}$&GTF&\cite{SheekeyMRD}\\
GTG&$q^s$&$a$&$\eta a^{q^h}$&$N(\eta)\ne(-1)^{nk}$&GTF&\cite{LuTrZh2015}\\
AGTG&$q^s$&$a$&$\eta a^{p^h}$&$N_p(\eta)\ne(-1)^{nk}$&GTF&\cite{Ozbudak1}\\
TZ&$q^s$&$a_0$&$\eta a_1$&$n$ even, $N(\eta)\notin \Box$&Hughes-Kleinfeld&\cite{TrZhHughes}\\
\hline
\end{tabular}
\end{center}
TG=Twisted Gabidulin, GTG=Generalised Twisted Gabidulin, AGTG=Additive Generalised Twisted Gabidulin. TZ = Trombetti-Zhou

Here for $n$ even we write $a=a_0+a_1\theta$ for unique $a_i\in \FF_{q^{n/2}}$, and $\Box$ denotes the set of squares in $\Fq$.

\begin{center}
{\bf Open Problem:}  Classify MRD codes contained in $\G_{k+1}$ and containing a code equivalent to $\G_{k-1}$.
 Classify MRD codes of dimension $nk$ contained in $\G_{k+1}$.
 Find exact conditions for when an element of $\G_{k+1}$ has rank $m-k$.
\end{center}
The latter of these is addressed in \cite{McGSh}.

\subsection{Puncturings of (twisted) Gabidulin codes}

The above outlined constructions all give MRD codes in $M_{n\times n}(\Fq)$, from which we can obtain MRD codes in $M_{n \times m}(\Fq)$ by puncturing and shortening. It is a difficult problem in general to ascertain if and when different puncturings of two codes are equivalent. For example, two puncturings of equivalent codes can be inequivalent, and puncturings of two inequivalent codes can turn out to be inequivalent. Recent work on puncturings of Gabidulin \cite{SchmidtPunc} and twisted Gabidulin \cite{CsSi} has made progress in this direction, though open questions still remain.

%

\subsection{MRD codes from Skew-polynomial Rings}

Further generalisations have been obtained in \cite{SheekeySkew}, by replacing $\sigma$-linearized polynomials with a quotient ring of a skew-polynomial ring.

\begin{proposition}
Let $\Fqn[t;\sigma]$ be a skew-polynomial ring, and $\overline{f}(y)$ an irreducible polynomial of degree $s$ in $\Fq[y]$, and $\eta\in \Fqn$ such that $N(\eta)\ne (-1)^{nks}$. Then the image of the set 
\[
\{f\in \Fqn[t;\sigma]:\deg(f)\leq sk, f_{sk} = \eta f_0\} 
\]
in the quotient ring
\[
\frac{\Fqn[t;\sigma]}{(\overline{f}(t^n))}\simeq M_{n\times n}(\FF_{q^s})
\]
is an MRD code of size $q^{nks}$. This includes examples of MRD codes not equivalent to any previous constructions. 
\end{proposition}

Taking $s=1$ returns us to the case of $\sigma$-linearized polynomials, so this construction contains the Twisted Gabidulin family. Taking $k=1$ and $\eta=0$ returns Petit's semifields. In general other values of $s,k,\eta$ give MRD codes and semifields not equivalent to any previous construction. We note that these codes are never $\FF_{q^{ns}}$-linear if $s>1$.

\subsection{Constructions for $d=n-1$, and Scattered Linear Sets}

Further recent constructions for the case of $\Fqn$-linear codes with $m=n=d+1$ have been obtained. This case has attracted particular attention due to the correspondence with linear sets on a projective line.


\begin{definition}
A linearized polynomial is called {\it scattered} if 
\[
\#\left\{\frac{f(x)}{x}:x\in \Fqn^*\right\} = \frac{q^n-1}{q-1}.
\]
\end{definition}
This definition is due to the fact that the $\Fq$-subspace of $\Fqn^2$ defined by
\[
U_f = \{(x,f(x)):x\in \Fqn\}
\]
is {\it scattered} with respect to the desarguesian spread; that is, it meets any space of the form $\{(x,\alpha x):x\in \Fqn\}$ in at most a one-dimensional space over $\Fq$. Scattered subspaces define {\it scattered linear sets}, which are objects in finite geometry which appear in many contexts. We refer to \cite{LavScat} for an overview of this.

In \cite{SheekeyMRD}, a connection between certain MRD codes and scattered linear sets on a projective line was made. This was extended in different ways in \cite{CsMaPoZu} and \cite{SheekeyVdV}. 


\begin{theorem}[\cite{SheekeyMRD}]
Equivalence classes of maximum scattered subspaces with respect to the Desarguesian $n$-spread in $V(2n,q)$ are in one-to-one correspondence with equivalence classes of $\Fqn$-linear MRD codes in $M_{n\times n}(\Fq)$ with minimum distance $n-1$
\end{theorem}

\begin{theorem}[\cite{CsMaPoZu}]
$\Fq$-linear MRD codes in $M_{\frac{sn}{2},n}(\Fq)$ with minimum distance $n-1$ can be constructed from scattered subspaces of $\Fq$-dimension $\frac{sn}{2}$ in $V(s,q^n)$.
\end{theorem}

A linearized polynomial $f$ over $\Fqn$ is said to be {\it exceptional scattered of index $t$} if $f(x^{q^{sn-t}})$ is scattered over $\FF_{q^{sn}}$ for infinitely many $s$.

\begin{theorem}[\cite{BaZh}]
The only linearized polynomial that is exceptional scattered with index $0$ is $x^q$. The only linearized polynomials that are exceptional scattered with index $1$ are those of the form $bx+ x^{q^2}$ with $N(b)\ne 1$.
\end{theorem}

Here we list the known examples of scattered polynomials, or equivalently the known $\Fqn$-linear MRD codes in $M_{n\times n}(\Fq)$ with minimum distance $n-1$. In each of the following, the code is represented as a two-dimensional $\Fqn$-subspace of linearized polynomials $\C = \pt{x,f(x)}$.

\begin{center}
\begin{tabular}{|c|c|c|c|c|}
\hline
Reference &$n$&$f(x)$&Conditions&Family\\
\hline
\cite{BlLa00}&any&$x^{\sigma}$&none&DG\\
\cite{LuPo1999}&any&$x^{\sigma}+\eta x^{\sigma^{-1}}$&$N(\eta)\ne 1$&GTG\\
\cite{CsMaPoZa}&6,8&$\eta x^{q^s}+x^{q^{s+n/2}}$ &$(n,s)=1$, $\eta^{q^{n/2}+1}\notin\{0,1\} $&-\\
\cite{CsMaZu}&6&$ x^q + x^{q^3}+bx^{q^5}$&$ b^2+b = 1, q \equiv 0, \pm 1 \mod 5$&-\\
\hline
\end{tabular}
\end{center}
DG=Delsarte-Gabidulin, GTG=Generalised Twisted Gabidulin

It was shown in \cite{CsZa} that for $n=4$, all $\FF_{q^4}$-linear MRD codes in $(\FF_{q^4})^4$ with minimum distance $3$ are equivalent to a Delsarte-Gabidulin or Twisted Gabidulin code. 

%
\begin{center}
{\bf Open Problem:} Do there exist $\FF_{q^n}$-linear MRD codes in $(\FF_{q^n})^n$ with minimum distance $n-1$ not equivalent to a Delsarte-Gabidulin or Twisted Gabidulin code for any $n$ odd? For any $n>10$ even?
\end{center}


\subsection{Other constructions, $m<n$}

Further constructions for MRD codes have been published in \cite{Puchinger}, \cite{HoMa}, with the latter giving constructions for $\Fqn$-linear codes. However difficulty remains in establishing newness in the case of $m<n$, due to the the difficulty in recognising puncturings and shortenings of Twisted Gabidulin codes.

\section{Classifications of MRD codes}

Classifications of MRD codes are only known in a few cases, most coming from classifications of semifields. We gather the known results here.
%
%
%

%

In this table we state the number of equivalence classes of $\Fq$-linear MRD codes in $M_{n \times m}(\Fq)$ with minimum distance $d$ for all known completed classifications. The first table lists theoretical results, while the second lists computational results. Note that the precise number of equivalence classes of generalised twisted fields can be found in \cite{Purpura}. Note also that for the case $m=n=d$, the number of equivalence classes corresponds to the number of isotopy classes, with the number of Knuth orbits in brackets. Results marked as new were calculated by the author using MAGMA.
\begin{center}
\begin{tabular}{|c|c|c|c|c|}
\hline
$n \times m$ &$q$&$d$&Classes&Reference\\
\hline
$2\times 2$&any&$2$&Field&\cite{Dickson1905}\\
$3\times 3$&any&$3$&Field/GTF&\cite{Menichetti1977}\\
$3\times 3$&any&$2$&dual of Field/GTF&\cite{Menichetti1977}\\
$n\times n$, $n$ prime&large enough&$n$&Field/GTF&\cite{Menichetti1996}\\
\hline
\end{tabular}
\end{center}

\begin{remark}
Note that by duality (Section \ref{sec:duality}), all results on constructions and classifications of semifields not only give immediate results for MRD codes with $m=n=d$, but also results for $m=n$, $d=2$. Thus for $n=2,3$, the problem of classifying $\Fq$-linear codes in $M_{n\times n}(\Fq)$ is completely solved. However if $q = p^e$ for some prime $p$ and integer $e>1$, the classification of $\Fp$-linear codes in $M_{n\times n}(\Fq)$ remains open for all $n$.
\end{remark}

\begin{center}
\begin{tabular}{|c|c|c|c|c|}
\hline
$n \times m$ &$q$&$d$&\#MRD Classes&Reference\\
&&& (Knuth Orbits)&\\
\hline
$3\times 4$&2&$3$&7&\cite{Kurz}\\
$3\times 4$&3&$3$&43&New\\
\hline
$4\times 4$&2&$4$&3(3)&\cite{Knuth1965}\\
$4\times 4$&2&$3$&1&New\\
{\color{blue}$4\times 4$}&{\color{blue}3}&${\color{blue}3}$&{\color{blue}5}&{\color{blue}\cite{ShBin}}\\
$4\times 4$&3&$4$&27 (12)&\cite{Dempwolff}\\
$4\times 4$&4&$4$&(28)&\cite{Rua2011a}\\
$4\times 4$&5&$4$&(42)&\cite{Rua2011a}\\
$4\times 4$&7&$4$&(120)&\cite{Rua2012}\\
\hline
$5\times 5$&2&$5$&6 (3)&\cite{Walker}\\
{\color{blue}$5\times 5$}&{\color{blue}2}&${\color{blue}4}$&{\color{blue}2}&{\color{blue}\cite{ShBin}}\\
$5\times 5$&3&$5$&23 (9)&\cite{Rua2011b}\\
\hline
$6\times 6$&2&$6$&332 (80)&\cite{Rua2009}\\
{\color{blue}$6\times 6$}&{\color{blue}2}&${\color{blue}5}$&{\color{blue}1}&{\color{blue}\cite{ShBin}}\\
\hline
\end{tabular}
\end{center}


\begin{remark}
\label{rem:table}
Representatives for equivalence classes of linear and nonlinear MRD codes for small parameters can be found at \cite{DempData}. Further data on subspace codes arising from MRD codes can be found at \cite{SubspaceCodes}.
\end{remark}

Here we state the number of equivalence classes of MRD codes (not necessarily additive) in $M_{n \times m}(\Fq)$ with minimum distance $m$ for all known completed classifications. We refer to \cite[Chapter 96]{Handbook} and references therein for all bar the case of $3\times 4$, for which we refer to \cite{Kurz}.
\begin{center}
\begin{tabular}{|c|c|c|c|c|}
\hline
$n \times m$ &$q$&Number\\
\hline
$2\times 2$&2&1\\
$2\times 2$&3&2\\
$2\times 2$&5&21\\
$2\times 2$&7&1347\\
\hline
$3\times 3$&2&1\\
$3\times 3$&3&7\\
\hline
$3\times 4$&3&37\\
\hline
$4\times 4$&2&8\\
\hline
%
\end{tabular}
\end{center}


Some classification results are known for $\Fq$-linear MRD codes in $M_{2\times 2}(\FF_{q^s})$ with minimum distance $2$.

\begin{theorem}[\cite{CAPOTR2004}]
Every $\Fq$-linear MRD code in $M_{2\times 2}(\FF_{q^2})$ with minimum distance $2$ is equivalent to one of four known constructions.
\end{theorem}

$\Fq$-linear MRD code in $M_{2\times 2}(\FF_{q^3})$ with minimum distance $2$ have been partially classified; see for example \cite{MaPoTr2007}.


\section{Ubiquity of MRD codes}

A fundamental question regarding MRD codes is their ubiquity; that is, if we choose a rank-metric code at random, what is the probability that it will be MRD? This question has drawn some attention in recent years, in particular asymptotic results; we collect the known results here. 

The nature of the results and open questions depends strongly on whether we are considering $\Fq$-linear or $\Fqn$-linear codes. To take a rather trivial example, the probability that a random one-dimensional $\Fqn$-linear code in $M_{n\times n}(\Fq)$ is MRD is precisely $|\GL(n,q)|/(q^{n^2}-1)$, or approximately $1-1/q$. However, in terms of equivalence classes, there are precisely $n$ classes, one for each rank, and so the  probability that a randomly chosen equivalence class of one-dimensional $\Fqn$-linear codes in $M_{n\times n}(\Fq)$ is MRD is $1/n$.

For $n$-dimensional $\Fq$-linear codes in $M_{n\times n}(\Fq)$, the probability that a randomly chosen example is MRD (i.e. a semifield spread set) is not known. Kantor conjectured the following (paraphrased from \cite{Kantor2006}).

\begin{Kantor}

(1) The number of equivalence classes of additive MRD codes in $M_{n\times n}(\Fp)$ with minimum distance $n$ is not bounded above by a polynomial in $p^n$.

(2) There is an exponential number of equivalence classes of additive MRD codes in $M_{n\times n}(\Fp)$ with minimum distance $n$.
\end{Kantor}

For other minimum distances, the following is known. We see a marked distinction between the cases of $\Fqm$-linear codes and $\Fq$-linear codes. Note that the statements are in terms of codes rather than equivalence classes.

\begin{theorem}{\cite{Genericity}}
The proportion of $k$-dimensional $\Fqn$-linear codes in $M_{n \times m}(\Fq)$ which are MRD approaches $1$ as $n \rightarrow \infty$.
\end{theorem}

\begin{theorem}{\cite{FerrersHeide},\cite{ByrneRavagnaniGen}}
The proportion of $mk$-dimensional $\Fq$-linear codes in $M_{n \times m}(\Fq)$ which are MRD is asymptotically at most $1/2$ if $q$ is large enough with respect to $m,n,k$, or if $m$ is large enough with respect to $q,n,k$.
\end{theorem}

\subsection{Data}
Here we collect data for the number of $\Fq$-linear MRD codes in $M_{n\times n}(\Fq)$ with minimum distance $n$, compared to the number of $n$-dimensional subspaces of $M_{n\times n}(\Fq)$. The number of MRD codes is calculated from the known classifications for isotopy classes of semifields, as listed in the table before Remark \ref{rem:table}, and their automorphism groups as sets of matrices (called {\it autotopy groups} in the semifield literature).
\begin{center}
\begin{tabular}{|c|c|c|c|c|c|}
\hline
$q$&$n$&\#Spaces&\#MRD&\#Classes&\#MRD Classes\\
\hline
$2$&$3$&$788035$&$192$&48&1\\
$2$&$4$&$9.1\times 10^{14}$&$2.7\times 10^{7}$&&3\\
$2$&$5$&$4.3\times 10^{30}$&$2.6\times 10^{14}$&&6\\
$2$&$6$&$5.2\times 10^{54}$&$3.8\times 10^{22}$&&332\\
\hline
$3$&$3$&$6.7\times 10^{8}$&$8.7\times 10^{5}$&&2\\
$3$&$4$&$1.7\times 10^{26}$&$2.6\times 10^{14}$&&27\\
$3$&$5$&$4.4\times 10^{52}$&$3.3\times 10^{23}$&&23\\
\hline
$5$&$3$&$5.0\times 10^{12}$&$4.5\times 10^{9}$&&4\\
\hline
\end{tabular}
\end{center}

There are 451584 $\FF_2$-linear MRD codes in $M_{4\times 4}(\FF_2)$ with minimum distance $3$, all equivalent to the Delsarte-Gabidulin code. There are $6.3\times 10^{18}$ $8$-dimensional subspaces of $M_{4\times 4}(\FF_2)$.

\begin{center}
{\bf Open Problem:} Extend this table so that a clearer pattern may emerge.
\end{center}

\section{Duality}\label{sec:duality}

Delsarte \cite{Delsarte1978} showed, using the theory of association schemes, that the rank distribution of a subspace of matrices determines the rank distribution of its {\it dual}, defined with respect to a particular symmetric bilinear form. This gives the rank-metric analogue of the well-known {\it MacWilliams identity} for codes in the Hamming metric \cite{MacWSl}. 

\subsection{MacWilliams-like identity}

\begin{definition}
Given an $\Fq$-linear code $\C\subseteq M_{n \times m}(\Fq)$, its {\it Delsarte dual} is the set
\[
\C^{\perp} = \{B:B\in M_{n \times m}(\Fq), \Tr(AB^T)=0 \quad \forall A\in \C\}.
\]
\end{definition}

We will denote the above bilinear form by $(A,B) = \Tr(AB^T)$. 

\begin{theorem}
\label{thm:MacW}
Suppose $\C\subseteq M_{n \times m}(\Fq)$ is $\Fq$-linear, with rank-distribution $r(\C) = (a_0,a_1,\ldots,a_m)$. Then the rank distribution $r(\C^\perp) =  (a_0^\perp,a_1^\perp,\ldots,a_m^\perp)$ of its Delsarte dual $\C^\perp$ is determined by
\[
a_j^\perp =  \frac{1}{|\C|}\sum_{i=0}^m a_i \sum_{s=0}^m (-1)^{j-s}q^{ns+{j-s\choose 2}}{ m-s\brack m-j}{m-i \brack s}.
\]
\end{theorem}
Delsarte first proved this \cite{Delsarte1978}. Gadouleau \cite{Gadouleau} proved a more concise statement, but as it requires the theory of $q$-polynomials (which are not to be confused with linearized polynomials), we do not include it here. Ravagnani \cite{Ravagnani} recently gave a more elementary proof of this, including useful recursive formulae, and a more general formulation in \cite{RavagnaniLattices}.

The following important theorem follows immediately from the above.
\begin{theorem}[\cite{Delsarte1978}]
\label{thm:dualMRD}
An $\Fq$-linear code $\C$ is MRD if and only if its Delsarte dual $\C^\perp$ is MRD.
\end{theorem}

\subsection{Other dualities satisfying the MacWilliams-like identities}

The bilinear form used by Delsarte to define duality is not unique with the property that it satisfies the MacWilliams-like identities. We show how this can occur, in particular when moving between different representations of rank-metric codes. If we are just interested in the rank-distribution of a dual, then it does not make much difference which of these forms we choose. However if we wish to speak of {\it self-dual} codes, as in \cite{NebeWillems}, then the choice of form does come in to play.

The following is well-known. We include a proof for illustrative purposes.
\begin{lemma}\label{lem:dualequiv}
For invertible matrices $X\in \GL(m,q)$, $Y\in \GL(n,q)$, it holds that
\[
(X\C Y)^\perp = X^{-T}\C^\perp Y^{-T}.
\]
\end{lemma}

\begin{proof}
By definition,
\begin{align*}
B\in (X\C Y)^\perp&\Leftrightarrow \Tr(XAYB^T)=0\forall A\in \C \\
&\Leftrightarrow \Tr(AYB^TX)=0\forall A\in \C \\
&\Leftrightarrow \Tr(A(X^TBY^T)^T)=0\forall A\in \C \\
&\Leftrightarrow X^TBY^T\in \C ^\perp\\
&\Leftrightarrow B\in X^{-T}\C ^\perp Y^{-T},
\end{align*}
proving the claim.
\end{proof}

In order to apply this in the various commonly used representations, we need to be careful to ensure that we are using a valid bilinear form. Moreover taking the dual of a punctured or shortened code requires particular care. A detailed examination of this can be found in \cite{ByrneRavagnani}.

Note that the trace function for matrices can be seen as a special case of the following definition of a trace function on tensors. Consider $V_1\otimes V_2$, and let $\beta_1$, $\beta_2$ be non degenerate bilinear forms on $V_1,V_2$ respectively. Then 
\[
(a\otimes b,c\otimes d)_{(\beta_1,\beta_2)}:= \beta_1(a,c)\beta_2(b,d),
\] 
and we extend this linearly to all elements of $V_1\otimes V_2$.

Taking $V_1 = \Fq^m$, $V_2 = \Fq^n$, and $\beta_i$ the usual dot product on $V_i$, then 
\[
\pt{A_1,A_2}_{\beta_1,\beta_2} = \Tr(A_{1}A_{2}^t),
\]
where we identify the tensor $a\otimes b$ with the matrix $ab^t$.

We consider now what happens if we were to choose different bilinear forms on $V_i$. We recall the following well-known fact.
\begin{proposition}
Let $V= \Fq^m$. Then any non degenerate bilinear form $\beta$ on $V$ can be expressed uniquely as
\[
\beta(u,v) = u^TBv,
\]
where $B$ is an invertible matrix in $M_{m\times m}(\Fq)$. 
\end{proposition}

\begin{proposition}
\label{prop:MacWbeta}
Let $V_1= \Fq^m, V_2=\Fq^n$, and suppose $\beta_i$ is a non degenerate  bilinear form on $V_i$. Then the  bilinear form on $M_{n \times m}(\Fq)$ defined by $(\beta_1,\beta_2)$ satisfies the Delsarte-MacWilliams identities.
\end{proposition}

\begin{proof}
Let $\beta_i$ have associated matrix $B_i$ with respect to the standard basis of $V_i$. Then
\begin{align*}
(ab^T,cd^t)_{(\beta_1,\beta_2)} &= \beta_1(a,c)\beta_2(b,d)\\
&= (a^tB_1 c)(b^tB_2 d)\\
&= (c^tB_1^t a)(b^t B_2 d)\\
&= \Tr ((c^tB_1^t a)(b^t B_2 d))\\
&= \Tr ((B_1^t ab^t)(B_2 dc^t))\\
&=\Tr((B_1^t ab^t)(cd^t B_2^t))\\
&= (B_1^t ab^t,cd^t B_2^t).
\end{align*}
Thus we have that $(A,B)_{(\beta_1,\beta_2)}  = (B_1^tA,BB_2^t)$. Then a similar argument to Lemma \ref{lem:dualequiv} shows that
\[
\C ^{\perp_{(\beta_1,\beta_2)}} = (B_1^t\C B_2)^\perp = B_1^{-1}\C B_2^{-t}.
\]
Hence $\C ^{\perp_{(\beta_1,\beta_2)}}$ is equivalent to $\C ^\perp$, and so has the same rank distribution, proving the claim.
\end{proof}

\subsection{Duality for linearized polynomials}

When $m=n$, we use the identification from Section \ref{sec:representations} between the tensor $a\otimes b\in \Fqn\otimes \Fqn$ and the linearized polynomial $a\tr(bx)$. Then taking the non degenerate bilinear form $\beta(a,b) := \tr(ab)$, where $\tr$ denotes the field trace from $\Fqn$ to $\Fq$, we get the bilinear form
\begin{align*}
(a\tr(bx),c\tr(dx)) &= \tr(ac)\tr(bd) = \tr\left(\sum_{i=0}^{n-1} (a\tr(bx))_i (c\tr(dx))_i\right).
\end{align*}
Extending this to all linearized polynomials, we get that
\[
(f,g) = \tr\left(\sum_{i=0}^{n-1} f_i g_i\right).
\]
Noting that the adjoint of $c\tr(dx)$ is $d\tr(cx)$, we can also see that
\[
(f,g) = \tr\left( (f\hat{g})_0 \right),
\]
where $\hat{g}$ denotes the adjoint of $g$ with respect to $\beta$, and further that
\[
(f,g) = \Tr(D_fD_g^t).
\]
We now verify that duality with respect to this form respects the MacWilliams-like identities.

\begin{proposition}
Let $\C$ be an $\Fq$-linear set of $\sigma$-linearized polynomials. Define
\[
\C^{\perp'} = \left\{g:\tr\left(\sum_{i=0}^{n-1} f_i g_i\right)=0 ~\forall f\in \C\right\}.
\]
Then $\C$ and $\C^{\perp'}$ satisfy the identities in Theorem \ref{thm:MacW}. Moreover, $\C$ is MRD if and only if $\C^{\perp'}$ is MRD.

Furthermore, if $\C$ is $\Fqn$-linear, then
\[
\C^{\perp'} = \left\{g:\sum_{i=0}^{n-1} f_i g_i=0 ~\forall f\in \C\right\}.
\]
\end{proposition}

\begin{proof}
The first two statements follow immediately by Proposition \ref{prop:MacWbeta} and Theorem \ref{thm:dualMRD}. The final statement follows from a simple consideration of dimension.
\end{proof}

%
%
%
%
%
%


%
%
%

\section{Symmetric rank-metric codes}

In this section we restrict to considering subspaces of symmetric matrices (or self-adjoint endomorphisms, self-adjoint linearized polynomials etc.). The problem becomes more difficult, due to the fact that the subgroup which preserves both rank and symmetry has more than one orbit on elements of each rank. Let $S_n(\Fq)$ denote the set of symmetric $n\times n$ matrices over $\Fq$. We can define an action of $\GL(n,q)$ on $S_n(\Fq)$ by
\[
A\mapsto X^TAX.
\]

\begin{remark}
Note that we do not use this action to define equivalence of subspaces of symmetric matrices, only to obtain bounds for such codes. The reason for this is that there exist $n$-dimensional subspaces of $S_n(\Fq)$ which are equivalent as rank-metric codes, but not equivalent under this action, even if we include field automorphisms. See \cite{PottZhou} and references therein.
\end{remark}

There are two orbits on elements of each rank; we will call them plus and minus type. For even rank, this corresponds to the well-known distinction between elliptic and hyperbolic quadrics. For this reason, it makes sense to define the {\it type distribution} of a code $\C\subset S_n(\Fq)$; that is, a string $(a_{i,\pm}:i \in \{0,\ldots,n\})$, where $a_{i,+}$ denotes the number of elements of $\C$ of rank $i$ and plus type.

\begin{theorem}[\cite{Schmidt}]
\label{thm:sym}
Suppose $\C\subset S_n(\Fq)$ has minimum distance $d$. 
If $\C$ is additive, then 
\[
|\C|\leq\left\{
\begin{array}{ll}
q^{\frac{n(n-d+2)}{2}}&n-d \textrm{ even}\\
q^{\frac{(n+1)(n-d+1)}{2}}&n-d \textrm{ odd}\\
\end{array}\right.
\]

If $d$ is odd and $\C$ is not necessarily additive, then
\[
|\C|\leq\left\{
\begin{array}{ll}
q^{\frac{n(n-d+2)}{2}}&n \textrm{ odd}\\
q^{\frac{(n+1)(n-d+1)}{2}}&n \textrm{ even}\\
\end{array}\right.
\]
If $d$ is even and $\C$ is not necessarily additive, then
\[
|\C|\leq\left\{
\begin{array}{ll}
q^{\frac{n(n-d+3)}{2}}\left(\frac{1+q^{-m+1}}{q+1}\right)&n \textrm{ odd}\\
q^{\frac{(n+1)(n-d+2)}{2}}\left(\frac{1+q^{-m+d-1}}{q+1}\right)&n \textrm{ even}\\
\end{array}\right.
\]
\end{theorem}
We note that when $d$ is even, the upper bound for non-additive codes is greater than that for additive codes. However it is not known whether there exist a non-additive code exceeding the bound for additive codes.
\begin{center}
{\bf Open Problem:} Do there exist non-additive symmetric codes larger than the best possible additive codes for even minimum distance?
\end{center}

We call an additive code meeting the upper-bound a {\it maximal additive symmetric rank distance code}. It is known that the bound for additive codes is tight in all cases.
\begin{theorem}[\cite{Schmidt},\cite{Cooperstein}]
The bounds for additive codes in Theorem \ref{thm:sym} are tight.
\end{theorem}

However, unlike in the case of arbitrary matrices, the rank and type distribution of an additive code meeting the bound is not uniquely determined. This was noted in \cite{Schmidt}, where examples of potentially different type distributions were given, though it was not known whether such distributions actually occur. This has been resolved by explicit examples to appear in \cite{GoSiSh}.

\begin{theorem}[\cite{GoSiSh}]
The rank distribution and type distribution of a maximal additive symmetric rank distance code are not completely determined by their parameters $q,n$, and $d$.
\end{theorem}

In the case of minimum distance $d=n$, the upper bound is $q^n$, which is identical to the Singleton-like bound, and so all maximal additive symmetric rank distance codes with minimum distance $n$ are in fact MRD codes, and thus correspond to quasifield. In the additive case they correspond to semifields. It is known that semifields defined by a subspace of symmetric matrices are obtained from semifields in which multiplication is {\it commutative}. 

\begin{theorem}[\cite{KantorComm}]
Equivalence classes of additive maximal additive symmetric rank distance codes in $M_{n\times n}(\Fq)$ with minimum distance $n$ are in one-to-one correspondence with isotopy classes defined by commutative semifields.
\end{theorem}
We note however that this correspondence is not direct; the spread set of a commutative semifield does not necessarily consist of symmetric matrices. One has to perform the semifield operation known as {\it transposition}, part of the {\it Knuth orbit} of a semifield, in order to obtain a set of symmetric matrices. Because of this issue, it is common in the literature on semifields to refer to {\it symplectic semifields} rather than commutative semifields. We note one important result here.

\begin{theorem}[\cite{BaBlLa}]
If $q$ is large enough, then every maximal additive symmetric rank distance code in $S_2(\Fq)$ with minimum distance $2$ is equivalent to one of a small number of known constructions.
\end{theorem}

In \cite{LaRo}, maximal additive symmetric rank distance code in $S_2(\FF_{2^4})$ and  $S_2(\FF_{3^4})$ were classified.

We can use the same bilinear form as before to define duality for $U\leq S_n(\Fq)$:
\[
U^{\perp} = \{B:B\in S_n(\Fq), \Tr(AB)=0 \quad \forall A\in U\}.
\]
In this case however, the rank distribution of $U$ does not determine the rank distribution of $U^\perp$. Instead we need to consider the type distribution. The following was shown by Schmidt in \cite[Theorem 3.1]{Schmidt}.

\begin{theorem}[\cite{Schmidt}]
Suppose $\C$ is an additive set in $S_n(\Fq)$. Then the type distribution of $\C^{\perp}$ is uniquely determined by the type distribution of $\C$.
\end{theorem}

%
%


\section{Related topics}

Many other properties of MRD code, and other types of rank-metric codes have been investigated recently. We list some of this research here.
\begin{itemize}
\item
Gabidulin codes which are self-dual with respect to Delsarte's bilinear form \cite{NebeWillems}.
\item
The {\it covering radius} of a rank-metric code \cite{ByrneRavagnani},
\item 
The {\it zeta-function} of a rank-metric code \cite{zeta}.
\item
Codes which are optimal with given {\it Ferrers diagram} \cite{FerrersEtzion}, \cite{FerrersZhang}, \cite{FerrersHeide}.
\item
Quasi-MRD codes, that is, codes which are optimal with a given dimension not divisible by $n$ \cite{QuasiMRD}.
\item
Almost MRD codes, that is, codes with minimum distance one less than the maximum possible \cite{DeLaCruz}.
\item
Non-additive MRD codes with $m<n$ \cite{DuSi}, \cite{CoMaPa}, \cite{OzNon}.
\item
Optimal codes of alternating matrices \cite{GoQu2009b} or hermitian matrices \cite{GoLaShVa}, \cite{Ihringer}, \cite{SchmidtHerm}. .
\item
Constant rank codes, that is subspaces of matrices in which every nonzero element has a fixed rank \cite{Gow}, \cite{SheekeyThesis}.
\end{itemize}

\subsection{Decoding and Cryptography}
\label{sec:crypto}

Rank-metric codes have been proposed for use in cryptography via a McEliece-type cryptosystem; see for example \cite{GabCrypto}. The key requirements to construct such a system are an efficient decoding algorithm, and a method to scramble the generator matrix of a code. To date Delsarte-Gabidulin codes have been used, due to their good error-correction capabilities (high minimum distance, efficient decoding algorithm), and the relatively small key-size ($\Fqn$-linearity means that a generator matrix is smaller than that for an $\Fq$-linear code).

There have been a variety of papers with the intention of improving systems based on Gabidulin codes, as well as papers showing weaknesses in some of these systems. For this reason, other MRD codes with efficient decoders are required in order to circumvent these attacks. Recent work on decoding Twisted Gabidulin codes can be found in \cite{TGabDecode}.

%


\begin{center}
{\bf Open Problem:} Construct new MRD codes with efficient decoding algorithms which do not possess the suspected weaknesses of the known examples.
\end{center}

\section{Appendix: Representations}
\label{sec:representations}

In the literature, various equivalent representations are used when working with rank-metric codes. Here we outline these, and give explicit ways of transferring between them.

To any $\Fq$-homomorphism from a vector space $V(n,q)$ to a vector space $V(m,q)$, we can associate each of the following:
\begin{itemize}
\item
a matrix $A\in M_{n \times m}(\Fq)$. The rank is the usual matrix rank;
\item
a vector $v\in (\Fqn)^m$; the rank is the dimension of the span of the entries over $\Fq$.
\item
a linearized polynomial $f\in \Fqn[x]$ of degree at most $q^{m-1}$; the rank is the rank as an $\Fq$-linear map on $\Fqn$.
\item
a Dickson matrix (if $m=n$); the rank is the usual matrix rank.
\item
a Moore matrix; the rank is the usual matrix rank.
\item
a tensor in $V(n,q)\otimes V(m,q)$; the rank is the usual tensor rank.
\end{itemize}

{\bf Between vectors and matrices}


Choose an $\Fq$-basis $\B = \{e_0,\ldots,e_{n-1}\}$ for $\Fqn$. Given a vector $v = (v_1,\ldots,v_m)\in(\Fqn)^m$, we expand each entry with respect to this basis; i.e . $v_j = \sum_{i=0}^{n-1}\lambda_{ij} e_i$. Then we identify the vector $v$ with the matrix
\[
A_{v,\B} = (\lambda_{ij})_{i,j} \in M_{n \times m}(\Fq).
\]
%

{\bf Between vectors and Moore matrices}

Let $v = (v_1,\ldots,v_m)\in \Fqn^m$. Define the {\it Moore matrix} $M_v$ by
\[
M_v := \npmatrix{v_1&v_2&\cdots&v_m\\v_1^\sigma&v_2^\sigma&\cdots&v_m^\sigma\\\vdots&\vdots&\ddots&\vdots\\v_1^{\sigma^{n-1}}&v_2^{\sigma^{n-1}}&\cdots&v_m^{\sigma^{n-1}}}.
\]
It is well known that $\rank(v) = \rank(M_v)$.

{\bf Between tensors and matrices}

Choose a basis $\B_1 = \{v_1,\ldots,v_n\}$ of $\Fq^n$ and $\B_2 = \{w_1,\ldots,w_m\}$ of $\Fq^m$. Then we identify the tensor $T = \sum_{i,j} T_{ij}w_i\otimes v_j$ with the matrix  $A_{T,\B_1,\B_2}$ whose $(i,j)$-entry is $T_{ij}$.  


{\bf Between linearized polynomials and vectors}

From a linearized polynomial $f$ and a fixed ordered string $\a = (\alpha_1,\ldots,\alpha_n)$ of $m$ elements of $\Fqn$, linearly independent over $\Fq$, we define the vector
\[
v_{f,\a} = (f(\alpha_1),f(\alpha_2),\cdots,f(\alpha_m)).
\]
Given $v$ and an ordered $m$-tubple $\a = (\alpha_1,\ldots,\alpha_m)$, there exists a unique linearized polynomial $f_{v,\a}$ of degree at most $q^{m-1}$ such that

\[
f_{v,\a}(\alpha_i) = v_i.
\]
Note that any other linearized polynomial $g$ with $g(\alpha_i)=v_i$ for all $i$ must be of the form $g = f_{v,\a}+h\circ m_\a$, where 
\[
m_\a := \prod_{\lambda \in \pt{\alpha_1,\ldots,\alpha_n} }(x-\lambda).
\]

{\bf From linearized polynomial to matrix}

Choose a basis $\B = \{v_1,\ldots,v_n\}$ of $\Fq^n$. Define $X_{ij}$ by
\[
f(e_j) = \sum_{i=1}^n X_{ij}e_i.
\]
Then we define the matrix of $f$ with respect to $\B$ as the $n\times n$ matrix with $(i,j)$-entry $X_{ij}$, and denote it by  $X_{f,\B}$. Clearly if $\B'$ is any other basis, then $X_{f,\B}$ and $X_{f,\B'}$ are similar.

Given an ordered $m$-tuple $\a = (\alpha_1,\ldots,\alpha_m)$ of elements of $\Fqn$, linearly independent over $\Fq$, then we have that
\[
A_{v_{f,\a},\B} = X_{f,\B}A_{\a,\B}.
\]

{\bf From tensor to linearized polynomial}

We identify the rank-one tensor $a\otimes b\in \Fqn\otimes \Fqn$ with the linearized polynomial
\[
f_{a\otimes b}(x) = a\tr(bx) = abx+ab^\sigma x^\sigma+\cdots + ab^{\sigma^{n-1}}x^{\sigma^{n-1}}.
\]
For any tensor $T= \sum_{i=1}^n a_i\otimes b_i$ we define
\[
f_T(x) = \sum_{i=1}^n a_i \tr(b_i x).
\]

{\bf From linearized polynomial to Dickson matrix}

From a linearized polynomial $f$ we define the {\it Dickson matrix}
\[
D_f = \npmatrix{f_0&f_1&\cdots&f_{n-1}\\
f_{n-1}^\sigma&f_0^\sigma&\cdots&f_{n-2}^\sigma\\
\vdots&\vdots&\ddots&\vdots\\
f_1^{\sigma^{n-1}}&f_2^{\sigma^{n-1}}&\cdots &f_0^{\sigma^{n-1}}}.
\]
Note that this is sometimes called an {\it autocirculant matrix}.

Then it holds that $D_fD_g = D_{f\circ g}$, where the composition is reduced mod $x^{\sigma^n}-x$. Furthermore, $D_{\hat{f}} = D_f^T$. Non-square Dickson matrices have been recently introduced in \cite{CsSi}.

{\bf Between Dickson matrices and matrices}

Let $\B = \{v_1,\ldots,v_n\}$ be an ordered $\Fq$-basis for $\Fqn$. 

Then the Dickson matrix of $f$ and the matrix of $f$ with respect to $\B$ satisfies
\[
D_f M_\B = M_\B X_{f,\B},
\]
and so $D_f$ and $X_{f,\B}$ are similar over $\Fqn$, with the transition matrix being the Moore matrix defined by $\B$. Thus we have the well-known fact that $\rank(f)=\rank(D_f)=\rank(X_{f,\B})$. Furthermore we have that $X_{f,\B}X_{g,\B}=X_{f\circ g,\B}$


%
%

\end{document}